\documentclass[11pt]{article}
\usepackage{color}
\usepackage{amsthm,amssymb,amsmath,latexsym}
\usepackage{graphicx}
  
\usepackage{amsthm}
\newtheorem{theorem}{Theorem}[section]
\newtheorem{lemma}[theorem]{Lemma}
\newtheorem{proposition}[theorem]{Proposition}
\newtheorem{corollary}[theorem]{Corollary}
\newtheorem{conjecture}[theorem]{Conjecture}

\theoremstyle{definition}

\theoremstyle{remark}

\theoremstyle{remark}
\newtheorem{example}[theorem]{Example}
\theoremstyle{remark}

\newcommand{\wt}{\widetilde}

\begin{document}
\title{A family of sand automata}

\date{}

\author{{\bf Nicholas  Faulkner}\footnote{This research was the content of Nicholas
 Faulkner's MSc thesis} \\ University of Ontario Institute of Technology\\
 {\bf Reem Yassawi}\footnote{Partially supported by an NSERC grant.}\\
Trent University, Peterborough, Canada\\
ryassawi@trentu.ca}

\maketitle

\begin{abstract} 
We study some dynamical properties of a family of two-dimensional
cellular automata: those that arise from an underlying one dimensional
sand automaton whose local rule is obtained using a latin square. We
identify a simple sand automaton $\Gamma$ whose local rule is algebraic, and classify this automaton as having equicontinuity points, but not being equicontinuous. We also show it is not surjective. We generalise some of these results to a wider class of sand automata.

\end{abstract}

\section{Introduction}\label{Introduction}
In \cite{cfm}, the authors introduce the family of {\em sand
  automata}: these are dynamical systems $\Phi: {\mathcal A}^{{\mathbb
    Z}^{d}}\rightarrow {\mathcal A}^{{\mathbb Z}^{d}}$, where $\mathcal
A$ is an countably infinite alphabet, that satisfy certain
constraints (see Section \ref{preliminaries} for all definitions).  In turn, with the appropriate topology $\mathcal T$ put
on ${\mathcal A}^{{\mathbb Z}^{d}}$, $( {\mathcal A}^{{\mathbb
    Z}^{d}}, \Phi)$ is topologically conjugate to a cellular automaton
$\Phi_{{{1}\choose{0}}} :S_{{{1}\choose{0}}} \rightarrow
S_{{{1}\choose{0}}}$ where $S_{{{1}\choose{0}}}\subset
\{0,1\}^{{\mathbb Z}^{d+1}}$ is a subshift of finite type. In
\cite{dgm} the authors ask  if any of  the cellular automata arising from sand automata  are chaotic. This is a particular instance of the more
general project of finding chaotic higher dimensional cellular
automata. In this article we study a family of sand automata and their dynamical properties.

In Section \ref{linear}, we define {\em linear} sand automata: these are automata whose local
rules are built using a group endomorphism $\phi: {\mathcal
  A}^{2r+1}\rightarrow {\mathcal A}$, where the finite alphabet
$\mathcal A$ is viewed as a cyclic group. We say that local rules are
`built', because, unlike cellular automata, the initial configuration
has to be {\em relativised} before the local rule is applied, and the
output of the local rule is {\em added} to the initial
configuration. In general we work exclusively with one dimensional sand automata (ie those
whose corresponding cellular automata act on two dimensional
configuration space), whose local rules have radius 1. We frequently work with one fixed sand automaton,
$\Gamma: {\mathcal A}^{\mathbb Z}\rightarrow {\mathcal A}^{\mathbb Z}$, whose local rule is built using the rule $\gamma:{\mathbb
  Z}_{5}^{3}\rightarrow {\mathbb Z}_{5}$ defined by $\gamma
({x_{-1}\,x_{0} \,x_{1}}) = x_{-1}+x_{1}$, which is the local rule for
the famous XOR cellular automaton (albeit on a different alphabet), and  whose dynamical properties have been studied extensively, for example in \cite{mm}, \cite{cfmm} and \cite{ino}.

Although we are interested in the dynamical properties of the cellular
automaton $\Phi_{{{1}\choose{0}}}$, in practice we often work with
$\Phi$, since all of the dynamical properties that we discuss are
topological invariants. In Section \ref{non_surjective}, we show that $\Gamma$ is not surjective -
recall that a cellular automaton which is not surjective cannot be
transitive. In Section \ref{surjective}, we identify a proper $\Gamma$-invariant subspace, ${\mathcal G}$,
and show that for points ${\bf x}$ in ${\mathcal G}$, computation of
$\Gamma^{n}(x)$ is simple, for all $n$. We identify other sand
automata for which ${\mathcal G}$ acts similarly - this accounts for
about 17\% of all radius one sand automata, all of which cannot be
transitive.

In Section \ref{equicontinuity} we show that $\Gamma$ is not equicontinuous, by showing the existence
of (many) {\em vertical inducing points}. In fact all radius one sand
automata whose local rule tables have, for some non-zero $m$, the value  $m$ appear in each column (or
each row),  have vertical inducing points, and so
are not equicontinuous. This means that at least $99\%$ of all sand
automata are not equicontinuous.  Despite not being
equicontinuous, we show that $\Gamma$ has equicontinuity points by
finding a {\em blocking word} for $\Gamma$. This completes the classification of $\Gamma$ according to the classification scheme in \cite{dgm}.

In Section \ref{local_rule_constant}, we generalise the definition of
a vertical inducing point to that of a {\em local rule constant
  point}. Both vertical inducing and local rule constant points have
easily computable $\Phi^{n}$ iterates. We identify a
$\Gamma$-invariant subspace $\mathcal G^{*}$, in which all points are
local rule constant. We define a subspace $\mathcal G'$ containing
$\mathcal G^{*}$, and conjecture that $\mathcal G'$ is an attractor for $\Gamma$, in that $\lim_{n\rightarrow \infty}d(\Gamma^{n}{\bf x}, \mathcal G')=0$ whenever ${\bf x}\in {\mathcal A}^{\mathbb Z}$.

\section{Preliminaries}\label{preliminaries}
\subsection{Notation }\label{notation}
If $\mathcal A$ is a countable alphabet, let $\mathcal A^{+}$ the set
of  finite concatenations of
letters from $\mathcal A$; elements of $\mathcal A^{+}$ are called {\em words}.  If $d$ is a natural number, we let $X=
\mathcal A^{\mathbb Z^{d}}:=\{\textbf{x} =
\{(x_{{\bf n}})_{{\bf n} \in \mathbb{Z}^d } :
x_{{\bf n}}\in \mathcal A \mbox{ for each }
{\bf n}\in \mathbb Z^{d}\}$; we call $X$ the
$d$-dimensional {\em configuration space } on the alphabet
$\mathcal{A}$, and elements of $X$ are called {\em configurations}. If
${\bf x}\,\in \mathcal A^{\mathbb Z}$, we sometimes write ${\bf x} =
\ldots x_{-1}.x_{0}x_{1} \ldots $. If
$\textbf{x} \in \mathcal A^{\mathbb Z}, \, m,n \in \mathbb{Z}, \, m \leq n$ let
$\textbf{x}_{[m,n]} := x_{m} x_{m+1} \ldots x_n$. We will be working in dimensions 1
and 2, and in dimension 2 only when $\mathcal A$ is finite.

We let $\mathcal{A}_r:=\{-r,\ldots ,r\}$, and, if $\mathcal A$ is given, let 
$\widetilde{\mathcal{A}}= \mathcal{A} \cup \{-\infty,\infty\}$.   
Loosely speaking if $d=2$ then \textbf{x} represents a
3-dimensional sandpile where $x_{{\bf n}}$ represents the
number of sand grains at ``location" ${\bf n}$.  If
$x_{{\bf n}} = \infty$ there is a {\em source} at location
${\bf n}$ and if $x_{{\bf n}} = -\infty$ there
is a {\em sink} at location ${\bf n}$. 

Let $L\subset {\mathcal A}^{\mathbb Z^{d}}$ be a finite set. If ${\bf
  c} = \{c_{\bf l}: c_{\bf l} \in \mathcal A, \mbox{ for } {\bf l}\in
\mathbb Z_{d}\cap L\}$, define the {\em cylinder set} $[{\bf c}]=\{
x:x_{\bf l} = c_{\bf l}, \mbox{ for } {\bf l} \in L\}$.  If
$\mathcal{A}$ is finite the Cantor topology - generated by the
cylinder sets, as $L$ and ${\bf c}$ vary - makes ${\mathcal A}^{\mathbb Z^{d}}$ compact.  If $\mathcal{A} =
\widetilde{\mathbb{Z}}$ this  is no longer the case.
  Because of this it is convenient to embed $ \wt{\mathbb
  Z}^{\mathbb Z} $ in $\{0,1\}^{\mathbb Z^{2}}$, and equip it with the
subspace topology $\mathcal T$ that it inherits from $\{0,1\}^{\mathbb
  Z^{2}}$: this is what is done in $\cite{dgm}$.  Define the map
$e:{\widetilde{\mathbb Z}}^{\mathbb Z} \rightarrow
\{0,1\}^{\mathbb{Z}^2}$ by \[(e(\textbf{x}))_{i,j} =
\left\{\begin{array}{rl} 1 & \textit{if } j \leq x_i \\ 0 & \textit{if
} j > x_i \\
\end{array} \right. \,.\]
In other words, the amount of sand $x_{i}$ at location $i$ in ${\bf x}$ is recorded in the  $i$th column of $e({\bf x})$, where there are ones for all column entries that are indexed by at most $x_{i}$, and zeros above entry  $x_{i}$.
If  $x_{i}= \infty$, the $i$-th column of $e({\bf x})$ is constantly 1,
 and similarly if  $x_{i}=-\infty$  the $i$-th column of $e({\bf x})$ is constantly 0.

The transformation $e:X\rightarrow e(X)$ is bijective and its image $e(X)$ is a subshift
of finite type, where the unique forbidden word is
$ \begin{array}{|c|c|} \hline 1 \\ \hline 0 \\ \hline
\end{array} $.
If $\{0,1\}^{\mathbb Z^{2}}$ is equipped with the Cantor topology then
 since $e(X)$ is closed in
$\{0,1\}^{\mathbb{Z}^2}$,  it is also compact.  Let $\mathcal T$
be the topology on $X={\wt {\mathbb Z}}^{\mathbb Z}$ inherited from the
subspace topology on $e(X)\subset \{0,1\}^{{\mathbb Z}^{2}}$.

\subsection{Dynamical systems}\label{dyn_sys}
We recall some basic definitions.  If $X$ is a topological space and
$F:X\rightarrow X$ be continuous, then $(X,F)$ is a {\em topological
  dynamical system}.  $(X,F)$ is {\em transitive} if for all nonempty
open sets $ U, V \subseteq X, \exists n \geq 0 $ such that $ F^{-n}(U)
\cap V \neq \emptyset$.  $(X,F)$ is {{\em equicontinuous}} at $ x
\in X$ if $ \forall \epsilon >0, \exists\delta >0, \forall y \in
B_{\delta}(x), \forall n > 0,d(F^n(x),F^n(y))<\epsilon $; and $(X,F)$
is {\em equicontinuous} if $F$ is equicontinuous at all points in $X$.
 $(X,F)$ is {\em sensitive to initial conditions} if $ \exists \epsilon
> 0, \forall x \in X, \forall \delta > 0, \exists y \in B_{\delta}(x),
\exists n > 0, d(F^n(x), F^n(y))>\epsilon$. Finally $(X,F)$ is {{\em
    positively expansive}} if $ \exists \epsilon > 0, \forall x \neq
y, \exists n \geq 0, d(F^n(x), F^n(y)) \geq \epsilon $. A point $ x \in X$ is {periodic} if
 $F^n(x)=x$ for some $n$; if $n=1$ then $x$ is {\em fixed}.  We call the smallest such $n$ the
period of x.  The point $ x\in X$ is {\em eventually periodic} if
$\exists p, n \in \mathbb{N} $ such that $F^{p+n}(x) = F^p(x)$; we
call $p$ the {\em preperiod}.  A dynamical system is \emph{ultimately
  periodic} if $\forall x \in X$, x is eventually periodic.

A dynamical system is {\em chaotic} if is sensitive, transitive and has a dense set of periodic points. In \cite{bbcds} it is shown that 
if $(X,F)$ is transitive and has a dense set of periodic points then $(X,F)$ is sensitive. 
 In \cite{cm}, the authors show that if $(X,F)$ is a  cellular automaton (see below for definitions) then  transitivity alone implies sensitivity.

\subsection{Cellular and sand automata}\label{def_san_automata}

Let $X={\mathcal A}^{\mathbb Z}$ where $\mathcal A$ is a finite
alphabet.  A {\em cellular automaton} $F:X\rightarrow X$ is a map
defined by a {\em local rule} $f:{\mathcal A}^{2r+1}\rightarrow
\mathcal A$, so that for each $n\in \mathbb Z$, $(F({\bf x}))_{n}=
f(x_{n-r},\ldots x_{n+r})$. The simplest example of a cellular
automaton is the {\em shift map} $\sigma:X\rightarrow X$, where
$(\sigma({\bf x}))_{n}=x_{n+1} $.  If $X={\mathcal A}^{\mathbb
  Z^{d}}$, a cellular automaton can be defined similarly: if
$L_{r}:=\{{\bf l}: {\bf l}=(l_{1},l_{2}),\,\,|l_{i}|\leq r\}$ is the
2-dimensional box of radius $r$, and $f: {\mathcal
  A}^{L_{r}}\rightarrow {\mathcal A}$ is a map, then a cellular
automaton $F:{\mathcal A}^{\mathbb Z^{2}} \rightarrow {\mathcal
  A}^{\mathbb Z^{2}}$ is defined as $(F({\bf x}))_{\bf n} = f(\{x_{\bf n+i}: {\bf i}\in L_{r}\}).$
There are two shift maps on
two-dimensional lattice spaces $X={\mathcal A}^{{\mathbb Z}^{d}}$: let
$(\sigma_H({\bf x}))_{{\bf n}}= x_{{\bf n}+(1,0)} $,
$(\sigma_V({\bf x}))_{{\bf n}}= x_{{\bf n}+(0,1)}$ be the horizontal and
vertical shifts respectively.  In \cite{h}, Hedlund showed that $F$ is
a cellular automaton if and only if $F$ is a continuous, shift
commuting map.

Let  $X=\wt{\mathbb Z}^{\mathbb Z}$. Like cellular automata, sand automata are  defined in \cite{cfm} as functions $\Phi: X \rightarrow X$ which have a local rule $\phi$, except that the output of the local rule is added to the original entry.
First, in \cite{cfm} the authors  define a sequence of ``measuring
instruments" of precision $r$.
If $n \in \mathbb{N}$ and $m \in \mathbb{Z}$ then the {\em measuring
  tool of precision $r$ at reference height $m$} is the function
$\beta_r^m: \widetilde{\mathbb{Z}} \rightarrow
\widetilde{\mathcal{A}_r}$ where \begin{equation}\beta _r^m(n) =
\left\{ \begin{array}{ll} + \infty & \mbox{ if } n > m+r, \\ -\infty
  & \mbox{ if } n < m-r, \\ n-m & \mbox{ otherwise.} \\
\end{array} \right. \end{equation}

 Let $\phi:\wt{\mathcal A_{r}}^{2r}\rightarrow \mathcal A_{r+1}$ be given. Define 
 \[\Phi(\textbf{x})_i = \left\{\begin{array}{ll} x_i + \phi(\beta_r^{x_i}(x_{i-r}), \ldots \beta_r^{x_i}(x_{i-1}),\beta_r^{x_i}(x_{i+1}), \ldots \beta_r^{x_i}(x_{i+r})) & \mbox{if } x_i \in \mathbb{Z} \\
x_{i} & \mbox{ if } x_i = \pm \infty \\ 
\end{array} \right.\]

Thus 
$\Phi(\textbf{x})_i$ differs from $x_i$ by at most $r+1$, and is a
function of the cylinder $x_{[i-r,i+r]}$.  Note that $\phi$ is
applied not to $x_{[i-r,i+r]}$ but to a relativised version of
this word.

As a book-keeping device we define a projection $\Pi: X \rightarrow
\{\widetilde{\mathcal{A}}^{2r}_{r}\}^{\mathbb{Z}}$ by $\Pi(\textbf{x})_i =
\beta_r^{x_i}(x_{i-r}), \ldots
\beta_r^{x_i}(x_{i-1}),\beta_r^{x_i}(x_{i+1}), \ldots
\beta_r^{x_i}(x_{i+r})$.  Clearly $\Pi$ is not injective. 
We refer to $\Pi(\textbf{x})_i$ as a {\em gradient tuple}, and 
later in this article we use $L_1,\ldots L_r$ to describe the gradients
to the left of the reference position and $R_1, \ldots R_r$ to
describe the gradients to the right of the reference position.

\begin{example}
  This is Example 12 in  \cite{cfm}, which emulates the 
 behaviour of the original model defined in
\cite{btw}.  Define a 1 dimensional
sand automaton $F$ whose local rule $\phi: \wt{\mathcal A_{1}}^{2}\rightarrow \mathcal A_{1}$
 is given by:
\[\phi(a,b) = \left\{ \begin{array}{ll} 1 & \mbox{ if } a = \infty,
  \, b \neq -\infty \\ -1 & \mbox{ if } a \neq \infty, \, b = -\infty
  \\ 0 & \mbox{ otherwise }\\
\end{array}   \right. \]
$F$ has the property that a grain of sand falls to the right (and only the right) if the right neighbour is at least 2 smaller.  If the number of sand grains in the initial configuration ${\bf x}$ is finite, then $F^{n}(x)$ is eventually fixed.
\end{example}

Define the {\em vertical map} $\rho:X\rightarrow X$ where 
\[\rho(\textbf{x})_i =   \left\{ \begin{array}{ll} x_{i}+1 & \mbox{ if } |x_{i}|<\infty,
   \\ x_{i} & \mbox{ if } |x_{i}|= \infty,\\ 
 \end{array}   \right. \]
and  say that $ \Phi $ is
\emph{vertical commuting} if $ \Phi(\rho(x)) = \rho (\Phi(x)) $.
Also, say that $ \Phi $ is \emph{infiniteness conserving} if $
\Phi(\textbf{x})_i = \pm \infty \Leftrightarrow x_i = \pm \infty
$. Note that all sand automata are shift commuting, vertical
commuting, and infiniteness conserving; in fact this characterises 
them, as shown in Theorem 17 of \cite{cfm}:
\begin{theorem}
 $\Phi: X \rightarrow X $ is a sand automaton if and only if $\Phi $ is continuous, shift commuting, vertical commuting and infiniteness conserving.
\end{theorem}

The number $r$ in the definitions of cellular and sand automata is called the {\em radius}.
{\em In this article we only consider sand automata of radius 1.}

\subsection{Modelling sand automata as cellular automata}
Using the injection $e:{\wt{\mathbb Z}}^{\mathbb Z}\rightarrow
 \{0,1\}^{{\mathbb Z}^{2}}$, 
we can transform a sand automaton into a 2-dimensional cellular automaton as done in \cite{dgm}. 
 Letting $S_{{{1}\choose{0}}} $ denote $e(X)$, 
 define $\Phi_{{{1}\choose{0}}} =e \circ \Phi \circ e^{-1}:S_{{{1}\choose{0}}}\rightarrow S_{{{1}\choose{0}}}$. With this notation, we have the following lemma, whose proof is straightforward:

\begin{lemma}\label{sand=cellular}

$\Phi_{{{1}\choose{0}}} $ commutes with both the vertical and horizontal shifts, and 
if $X$ is endowed with the topology $\mathcal T$, then $(X,  F)\cong (S_{{{1}\choose{0}}}, \Phi_{{{1}\choose{0}}})$.

\end{lemma}

\subsection{Linear sand automata}\label{linear}
If $|\mathcal A|=2k+1$, then $\mathcal A$ can be viewed as the
additive group $\mathbb Z_{2k+1}= \{-k, \ldots ,0, \ldots
,k\}$. Recall that 
$F:{\mathbb Z}_{2k+1}^{\mathbb Z}\rightarrow {\mathbb Z}_{2k+1}^{\mathbb Z}$ is a {\em linear} cellular
automaton if its local rule $f:\mathbb Z_{2k+1}^{2r+1}\rightarrow \mathbb
Z_{2k+1}$ is a group homomorphism. Thus $f(\textbf{x}_{[i-r,i+r]}) =
a_0x_{i-r}+a_1x_{i-r+1}+ \ldots +a_{2r+1}x_{i+r}$ where $a_i$ and
$x_i$ $\in \mathbb{Z}_{2k+1}$ for  $0\leq i\leq 2r+1$. The topological properties of
$F$ depend on the coefficients $a_{i}$: for example if the $a_{i}'s$
are relatively prime, then $F$ is topologically transitive, so that many
linear cellular automata are chaotic.

Next we describe how to define a sand automaton using a linear cellular
automaton. Recall we only consider sand automata of radius 1. In this case  we
can display the local rule in terms of a \emph{local rule table}. A
local rule table has a row for each possible left gradient $L$ and a
column  for each possible right gradient $R$.  Entry $(L,R)$ of the table is $\phi(L,R)$. For example, Figure \ref{localruletableexample} is  the local rule table for the sand automaton $\Omega$ defined in Section \ref{def_san_automata}.
The local rule table is applied after each gradient
pair is calculated.  For example, if  $ \textbf{x} = \ldots,4,\cdot
2,1,\ldots $ then $(\Pi(x))_{0} = {{\infty}\choose{-1}}$,   so $L=\infty$ and 
 $R=-1$, so that $(\Omega({\bf x}))_{0} = x_{0}+1=3$.

\begin{figure}[htb]
\begin{center}
\begin{tabular}{c|c|c|c|c|c}
$L \backslash R $&$-\infty$&-1&0&1&$\infty$ \\ \hline
$-\infty$&-1&0 &0 &0 &0  \\ \hline  
-1& -1& 0&0 & 0&0 \\ \hline
0& -1& 0& 0&0& 0 \\ \hline
1& -1&0 &0& 0&  0\\ \hline
$\infty$& 0&1&1 &1 &1  \\ \hline
\end{tabular} 
\end{center}
\caption{The local rule table for the sand automaton $\Omega$.}
\label{localruletableexample}
\end{figure}

We now define {\em linear } sand automata of radius $r$. Recall that
to define $(F({\bf x}))_{0}$
 first we relativise, sending $(x_{-1}, x_{0}, x_{1})$ to 
$(\beta^{x_{0}}_1(x_{-1}),\beta^{x_{0}}_1(x_{1}))\in 
\widetilde{\mathcal{A}}_1^{2}$. 
Let $f:\widetilde{\mathcal{A}}_1
\rightarrow \mathbb{Z}_{5}$ be the bijection defined by

\[( f(x))_i = \left\{ \begin{array}{ll}
2 & \mbox{ if } x_i  = \infty \\
1 & \mbox{ if } x_i  = 1 \\
0 & \mbox{ if } x_i  = 0 \\
-1 & \mbox{ if } x_i  = -1 \\
-2 & \mbox{ if } x_i  = -\infty  
\end{array}
\right.\] 

Given a group homomorphism $\phi^*:\mathbb{Z}_{5}^{2} \rightarrow
\mathbb{Z}_{5}$, we say that a sand automaton $\Phi$ is a {\em linear}
sand automaton if the local rule $\phi = \phi^*\circ  f\circ
\Pi$.

\begin{example}\label{basic_example}

Let $\gamma_{*} : \mathbb{Z}_{5}^{2} \rightarrow \mathbb{Z}_{5} $ be defined by  $\gamma_{*}(x,y) = x \oplus y$. 
\end{example}

Let $\Gamma $ be the linear sand automaton with local rule $\gamma = \gamma_{*} \circ f\circ \Pi$.    
Then $ \Gamma $ is a radius 1 linear sand automaton. The local rule table in 
Figure \ref{specificlocalruletable}
corresponds to the group homomorphism $\gamma_{*}$.  Note that here the rows and columns are indexed by ${\mathbb  Z}_{5}$.

\begin{figure}[htb]
\begin{center}
\begin{tabular}{c|c|c|c|c|c}
$L \backslash R $&-2&-1&0&1&2 \\ \hline
-2&1&2 &-2 &-1 &0  \\ \hline  
-1& 2& -2& -1&0 & 1 \\ \hline
0& -2& -1& 0&1& 2 \\ \hline
1& -1& 0&1& 2& -2 \\ \hline
2& 0&1&2 & -2&-1  \\ \hline
\end{tabular} 
\end{center}
\caption{The local rule table for $\Gamma$}
\label{specificlocalruletable}
\end{figure}

In this article we will often be working with $\Gamma$, even though we
are interested primarily in $\Gamma_{{{1}\choose{0}}}$ (see Lemma
\ref{sand=cellular}). For, $\Gamma_{{{1}\choose{0}}}$ has radius $r=7$
(in fact the local neighbourhood can be a 7x3 rectangle) grid, so it is often more practical to work with the radius one $\Gamma$. Figure \ref{SAtoCA2} contains  the local rule table for $\Gamma_{{{1}\choose{0}}}$.

\section{Non-surjectivity of $\Gamma$}\label{non_surjective}
A non-surjective sand
automaton $F$ has {\em Garden of Eden states}:  these are configurations that have no $F$-pre-image.
A non-surjective cellular automaton  cannot be chaotic, since it cannot be transitive. The surjectivity of sand automata is shown to be undecidable in \cite{cfm}.
In this section we show that the sand automaton $\Gamma$ is not surjective, and generalise this result to some other one dimensional sand automata.

Recall that a sand automaton $\Phi:X\rightarrow X$ is {\em surjective
  on a set } $Y\subset  X$ if for each $ {\bf y} \in Y,$ 
$\Phi({\bf y'}) = \textbf{y}$ for some ${\bf y'} \in Y$.
 A configuration is \emph{finite} if all configuration entries are finite, and only finitely many entries are non-zero; let $\mathcal F$ denote all such points. Similarly, let $\mathcal P$ denote the set of all $\sigma$-periodic configurations, all of whose entries are finite.
In \cite{cfm} (Proposition 3.14), the following was shown:

\begin{lemma}\label{surjectivity1}[\cite{cfm}]
Let $\Phi:X\rightarrow X$ be a one-dimensional sand automaton. Then
\begin{enumerate} 
 \item
$\Phi$ is surjective on $\mathcal P$ if and only if $\Phi$ is surjective, and
\item
If $\Phi$ is surjective on $\mathcal F$, then $\Phi$ is surjective.
\end{enumerate}
\end{lemma}

We show that  $\Gamma$ is not surjective by first showing
 that there is a word which has
no predecessor word under $\Gamma$.  
In the following proof the notation
$[n]_5$ denotes the projection of $ n \in \mathbb{Z}$
to $\mathbb{Z}_5$.
\begin{lemma}\label{nonsurjectivity}
Let $ \textbf{w} = (100,\, 3,\, 2,\, 100 )$.  Then there is no word \textbf{y}  such that $\Gamma(\textbf{y}) = \textbf{w}$.
\end{lemma}
\begin{proof} 
Suppose that  $ \textbf{y} = y_0\, y_1\, y_2\, y_3\, y_4\, y_5 $ is a word of length 6 such that  $\Gamma(\textbf{y}) = \textbf{w}$. Then 
\[ 98 \leq y_1 \leq 102, \,\,\,
1 \leq y_2 \leq 5, \,\,\,
0 \leq y_3 \leq 4, \mbox{ and  }
98 \leq y_4 \leq 102\, .\,\,\,\]
This implies that $\beta_1^{y_2}(y_1) = \beta_1^{y_3}(y_4) = \infty$, ie $f(\beta_1^{y_2}(y_1)) =f(\beta_1^{y_3}(y_4))=f(\infty) =2$.
Next under the action of $\Gamma$ we have: 
\begin{equation}
y_2 +   f(\beta_1^{y_2}(y_1)) + f(\beta_1^{y_2}(y_3)) =y_2 + 2 +  f(\beta_1^{y_2}(y_3)) = 3 
\label{321one}
\end{equation} 
\begin{equation}
y_3+  f(\beta_1^{y_3}(y_4))+  f(\beta_1^{y_3}(y_2))=y_3+ 2 + f(\beta_1^{y_3}(y_2))=2
\label{321two}
\end{equation}
Projecting (\ref{321one}) and (\ref{321two}) into $\mathbb{Z}_5$, we have 
\begin{equation}
[y_2]_5 \oplus 2 \oplus [y_3-y_2]_5 = -2
\label{321three}
\end{equation}
\begin{equation}
[y_3]_5 \oplus 2 \oplus [y_2-y_3]_5 = 2
\label{321four}
\end{equation} 
Adding (\ref{321three}) and (\ref{321four}) we get 
\begin{equation}
[y_2]_5\oplus [y_3]_5=1
\label{321five}
\end{equation}
Thus the only possibilities for $[y_2]_5$ and $[y_3]_5$ are given by:\\
\begin{center}
\begin{tabular}{c|c|c|c|c|c}
$[y_2]_5 $ / $[y_3]_5 $&-2&-1&0&1&2 \\ \hline
-2&x& & & &  \\ \hline  
-1& & & & &x \\ \hline
0& & & &x&  \\ \hline
1& & &x& &  \\ \hline
2& &x& & &  \\ \hline
\end{tabular} \\ \vspace{10pt}
\end{center}
This implies that the only possibilities for $y_2,y_3$ are: \\ \vspace{10pt}
\begin{center}
\begin{tabular}{c|c|c|c|c|c}
$y_2$ / $y_3$ &3&4&0&1&2 \\ \hline
3&x& & & &  \\ \hline  
4& & & & &x \\ \hline
5& & & &x&  \\ \hline
1& & &x& &  \\ \hline
2& &x& & &  \\ \hline
\end{tabular} \\ \vspace{10pt}
\end{center}
Each of these cases implies a contradiction to Equation (\ref{321five}).

%\begin{enumerate}
%\item
%If $(y_2,y_3)=(3,3)$, then  
%$ 3 = y_2 + 2 \oplus 0 = y_2 + 2 $, ie $ y_2 = 1$, a contradiction
%\item 
%If  $(y_2,y_3)=(4,2)$ then $2 = y_3 + 2 \oplus 2 = y_3 + 1 $, ie 
%$ y_3 = 1, $ a contradiction. 
%\item
%If $(y_2,y_3)=(5,1)$, then  
%$3 = y_2 + 2 \oplus -2 = y_2 + 0 $, ie  
%$\Rightarrow y_2 = 3 $ a contradiction.  

%\item If  $(y_2,y_3)=(1,0)$ then
% $3 = y_2 + 2 \oplus -1 = y_2 +1 $, ie
% $y_2 = 2 $, a contradiction. 
%\item If  $(y_2,y_3)=(2,4) $, then 
% $ 2 = y_3 + 2 \oplus -2 = y_3 + 0 $, ie 
% $ y_3 = 2 $, a contradiction.
%\end{enumerate}
% Therefore there is no word \textbf{y} such that $\Gamma(\textbf{y}) = \textbf{x}$.
\end{proof}

\begin{proposition}\label{FPsurjectivity}
$\Gamma$ is not $\mathcal P$-surjective, so that $\Gamma$ is not surjective.

\end{proposition}

\begin{proof}
Let $ {\bf x}$ be the periodic configuration ${\bf x} = \overline{100, \, 3, \, 2, \, 100,} \cdot \overline{100, \, 3, \, 2, \, 100}$.  Lemma \ref{nonsurjectivity} implies that there  does not exist $ {\bf y}$ such that $ \Gamma({\bf y}) = {\bf x}$. Lemma \ref{surjectivity1} implies that $\Gamma$ is not surjective.
\end{proof}
\subsection{Surjective subsets}\label{surjective}
While $\Gamma$ is not surjective, in this section we identify a
proper, closed, $ \Gamma$-invariant subspace, $\mathcal G$.  We then
identify sand automata $\Phi$ for which $\mathcal G$ is also
$\Phi$-invariant.  Let
\[{\mathcal G} := \{\textbf{x}: |x_i - x_{i-1}| \geq 2, \, \forall i \}\, .\] 
We shall show that 
 each member of $\mathcal G$  has a $\Gamma$-predecessor (though not
necessarily in $\mathcal G$).  We will  explain in geometric
terms how $\Gamma$ acts on $\mathcal G$, identify subsets of $\mathcal G$ that are a possible attractor for $\Gamma$,  and generalize
these properties to other sand automata.   

 It is clear that
$        f  \circ \Pi ({\mathcal G}) \subset
\{{{2}\choose{2}},{{-2}\choose{2}},{{2}\choose{-2}},{{-2}\choose{-2}}\}^{\mathbb{Z}}$,  and that $f \circ \Pi ({\mathcal G})$ is the 
 subshift of finite type whose  transition graph
$G$ is shown in  Figure \ref{graph1}. Re-label ${{2}\choose{2}} = -1,
{{-2}\choose{2}}=0^-, {{2}\choose{-2}}=0^+, {{-2}\choose{-2}}=1$, and let $Y_{G}$ denote the image of 
$f\circ \Pi(\mathcal G)$ under this labelling. Let
$\{-1,0,1\}^{\mathbb{Z}} $ be the full shift on three letters and let
$p:\{0^-,0^+,1,-1\}\rightarrow \{-1,0,1\} $ by defined by
$p(0^-)=p(0^+)=0,$ $p(1)=1$, $p(-1)=-1$; then p is a radius 0  local rule for the cellular automaton $P:Y_{G}\rightarrow
\{-1,0,1\}^{\mathbb{Z}}$. With this notation, the following lemma is straightforward.
\begin{figure}[h]
\centering
\centerline{\includegraphics[scale=0.2]{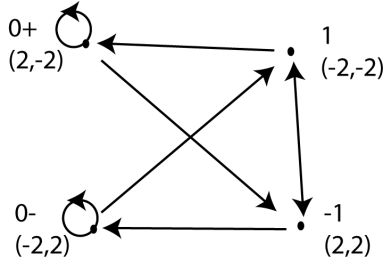}}
\caption{This graph $G$ defines a subshift on its bi-infinite paths.}
\label{graph1}
\end{figure}
\begin{lemma}\label{Gsurjective1}
If ${\bf g} \in {\mathcal G} $ then $\exists {\bf y} \in Y_{G}$ such that 
$\Gamma({\bf g}) = {\bf g} + P({\bf y})$.
\end{lemma}

\begin{table}[h]
\begin{center}
\begin{tabular}{c|c|c}
$L \backslash R$ &-2&2 \\ \hline
-2&1&0    \\ \hline  
2 & 0&-1 \\ \hline
\end{tabular}
\end{center}
\caption{The local rule table for $\Gamma$ restricted to $\mathcal G$}
\label{smalladd}
\end{table} 

On $\mathcal G$ the local rule table for $\Gamma$ can be compressed to Table \ref{smalladd}. 
  Call
3-tuples such that $f\circ \Pi (a,b,c) = {{-2}\choose{-2}}$
\emph{peaks} (centred at $b$) and similarly $f\circ \Pi
(a,b,c) = {{2}\choose{2}}$ \emph{valleys} (centred at $b$).  Also label
3-tuples such that $f\circ \Pi (a,b,c) = {{-2}\choose{2}}$
\emph{up-slopes} (centred at b) and $f \circ \Pi (a,b,c) =
     {{2}\choose{-2}}$ \emph{down-slopes} (centred at b).  
With this labelling of gradient tuples as {\em geographical features}, the action  of $\Gamma^n$ is easily described, once we have knowledge of how $\Gamma$ acts.

\begin{proposition}\label{Gsurjective2}
If ${\bf g} \in {\mathcal G}$ and ${\bf y} \in Y_{G}$ are  such that $\Gamma({\bf g}) = {\bf g} + {\bf y}$, then $\Gamma^n({\bf g}) = {\bf g} + n{\bf y}$   for
 all $ n $ in $ \mathbb{N}$.
\end{proposition}

\begin{proof}
Here we claim that all of the geographical
features are preserved under $\Gamma$.  If we show this then the
proposition follows.  First note that according to Table \ref{smalladd},
$|(\Gamma(g))_{n}- g_{n}|\leq 1$ for each $n$, whenever $g\in \mathcal
G$. Let $(g_{-1},g_{0},g_{1})=(a,b,c)$; we show that geographical
features are preserved at $\Gamma(g)_{0}$; the cases at  $\Gamma(g)_{n}$ for $n\neq 0$ are identical.
\begin{enumerate}
\item If $(a,b,c)$
is a valley centred at $b$,  then $(\Gamma(g))_{0}= g_{0}-1$  (see Table
\ref{smalladd}). Since  for $n=\pm 1$, $(\Gamma(g))_{n} -g_{n}$ is at least -1, then the valley at $b$ is mapped to another valley
centred at $(\Gamma(g))_{0}$. The case where $(a,b,c)$ is a peak is
similar.  \item If $(a,b,c)$ is a down-slope centred at
$b$, then we either have a peak or a down-slope centred at a.  Thus
$(\Gamma(g))_{-1}\geq g_{-1}$.
  Similarly there is either a valley
or a down-slope centred at $c$, so $(\Gamma(g))_{1}\leq g_{1}$. This means that 
a down-slope centred at $g_{0}$ is mapped to a down-slope centred at 
$(\Gamma(g))_{0}$. The case where $(a,b,c)$ is an up-slope is similar.
\end{enumerate}
\end{proof}

Note that the proof of Lemma \ref{Gsurjective2} shows that
$\gamma({\mathcal G}) \subset {\mathcal G}$: each geographical feature
under the action of the sand automaton can only become more pronounced
or stay the same.  However $\Gamma:{\mathcal G} \rightarrow {\mathcal
  G}$ is not surjective.  Consider ${\bf g}=\overline{0,3}, \cdot
\overline{0,3}$, so that $\Gamma({\bf g}) = {\bf g} + \overline{-1,1},
\cdot \overline{-1,1}$.  If $\Gamma^{-1}({\bf g}) $ contained  an element ${\bf
  g}^{\prime}$ in ${\mathcal G}$, then $ {\bf g} = {\bf
  g}^{\prime} + \overline{-1,1}, \cdot \overline{-1,1}$ by Lemma
\ref{Gsurjective2}.  So ${\bf g}^{\prime} = \overline{1,2}, \cdot
\overline{1,2} \notin {\mathcal G}$,  a contradiction.  The
next lemma tells us that although $\Gamma$ is not surjective on
$\mathcal G$, all configurations in $P(Y_{{\mathcal G}}) $ are used
when determining $\Gamma({\mathcal G})$.

\begin{lemma}
If ${\bf y} \in P(Y_G)$ then $\exists {\bf g} \in {\mathcal G} $ such that $\Gamma({\bf g}) = {\bf g} + {\bf y}$.
\label{Gsurjective3}
\end{lemma}

\begin{proof}
First find an element ${\bf y^*} \in
\{{{2}\choose{2}},{{-2}\choose{2}},{{2}\choose{-2}},{{-2}\choose{-2}}\}^{\mathbb{Z}}$
with $P( {\bf y^*}) ={\bf y}$. Then  choose $g_0$ arbitrarily and
follow the instructions given by ${\bf y^*}$ to specify $g_{-1}$ and
$g_1$ - for example, if $y^*_0 = {{2}\choose{-2}} $ then choose $g_1
\leq g_0 - 2 $ and $g_{-1} \geq y_0 + 2$.  Continue this process, using
the gradients at locations -1 and 1 to specify $g_{-2}$ and $g_{2}$, moving outwards from the central cell.
The result follows by induction.
\end{proof}

The proofs in this section for the sand automaton $\Gamma$
 relied on the preservation of certain
geographical features when restricted to the set of configurations
${\mathcal G}$. 
We now show that knowledge of  the entries $\alpha$, $\beta$, $\delta$ and $\lambda$ noted in Table
\ref{alphabetatable}, is sufficient to extend these results.

\begin{table}[h]
\begin{center}
\begin{tabular}{c|c|c|c|c|c}
$L \backslash R $&$-\infty$&-1&0&1&$\infty$ \\ \hline
$-\infty$&$\alpha$& & & &$\beta$ \\ \hline  
-1&& & & & \\ \hline
0& & & & & \\ \hline
1& & & & & \\ \hline
$\infty$& $\delta$& & & & $\lambda$\\ \hline
\end{tabular} 
\end{center}
\caption{The positions in the local rule table that are used to create peak/valley preserving sand automata.}
\label{alphabetatable}
\end{table}

\begin{theorem}
Let the radius one sand automaton $\Phi$ have local rule table as in Table \ref{alphabetatable}. Then $\Phi$ is peak preserving if and only if $\alpha \geq \beta,\delta,\lambda$ and $\Phi$ is  valley preserving if and only if $\lambda \leq \beta,\delta,\alpha$.
\end{theorem}   

\begin{proof}
Assume that $\Phi$ is the sand automaton described above. We use the notation
 $x_{[n-1,n+1]} = (a,b,c) $ and $(\Phi({\bf x}))_{[n-1,n+1]}= (a^{*},b^{*}, c^{*})$.
We prove the statement concerning peak preservation - the proof of the statement for valleys is similar.

  Suppose that  $\alpha \geq
\beta,\delta,\lambda$.  Let $(a,b,c)$ be a peak centred at b.  Then $b^{*}= b + \alpha$.
 (see Table \ref{alphabetatable}).  Since
$\alpha$ is the largest increase when restricted to $\mathcal G$, then $a^{*}, \,\, c^{*} \leq a+\alpha,\, c+\alpha$.
Thus peaks are mapped
to  peaks.  Conversely suppose that  $\Phi$ is peak preserving, and 
suppose that $\alpha <
\beta$. Then the configuration ${\textbf g} = \ldots-2, \, 0\, \cdot\,  2, \, 0
\ldots$, that has a peak at the central position and an up-slope at
the $-1$-st position, does not map the central peak to another peak.
This is due to the fact that $(\Phi({\textbf x}))_0 = 2 + \alpha$ and
$(\Phi({\textbf x}))_{-1} = 0 + \beta$, where $\alpha < \beta$.
Therefore $(\Phi({\textbf x}))_0 - (\Phi({\textbf x}))_{-1} > -2$.  Similar
configurations can be constructed when $\alpha < \delta$ or $\alpha <
\lambda$ using a valley or a down-slope at the $1$
position.

\end{proof}

 For the results of  Lemma \ref{Gsurjective2} to hold, which would also imply that 
$\mathcal G$ is $\Phi$ invariant,
it must be both peak and valley preserving as well as 
 being up-slope and down-slope preserving. 
This implies that we are assuming $\alpha \geq \beta, \delta \geq \lambda$. The next lemma demonstrates that  peak and valley preservation implies up-slope/ down-slope preservation.
\begin{proposition}
Let $\Phi$ be a radius sand automaton with local rule table as in Table 
\ref{alphabetatable}.  If $\Phi$ is both peak and valley preserving
then it is also up-slope and down-slope preserving,  so that  $\mathcal
G$ is $\Phi$ invariant and not surjective.
\label{invariantgeneral}
\end{proposition}

\begin{proof}
Let $\Phi$ be both valley and peak preserving. Then $\alpha \geq \beta, \delta \geq \lambda$. 
We show that down-slopes are mapped to down-slopes, the up-slope case being similar.
Suppose there is a down-slope centred at  $x_{n}=b$, where 
$x_{[n-1,n+1]} = (a,b,c)$.  
Then $(\Phi(x))_{n} = b +\beta$. We either have a peak or a down-slope
centred at $a$.  Thus$(\Phi(x))_{n-1} \geq  a +\beta$.
Similarly there is either a valley or a down-slope centred at c.
Therefore $(\Phi(x))_{n+1} \leq  c +\beta$.  Thus a down-slope centred
at $b$ is mapped to a down-slope centred at the image of $b$.
\end{proof}

 If we count
the number of sand automata that are both peak and valley preserving
by fixing both $\alpha$ and $\lambda$ and then counting all 
$\delta, \beta$ that do not violate the conditions $\alpha \geq
\beta,\delta \geq \lambda $, then  in total there are $ 105$ radius one peak and valley preserving automata.
There are in total $5^4=625$ possible choices for
$\alpha,\beta,\delta,\lambda$; thus  the set of sand automata
which is invariant on the set $\mathcal G$ represents $\frac{105}{625}
= 0.168$ of the total radius 1 sand automata.

 % The cases are according to the value of
%$\alpha$.
%\begin{enumerate}
%\item Suppose $\alpha = 2$. 
%If $ \lambda = -2$, there are no constraints on $\beta$ and $\delta$; thus there are $5^2$ choices.
%If $\lambda = -1$, then $\beta \,\, \delta\neq 2$ .  So there are $4^2$ choices. 
%If $\lambda = 0$ then using the same argument there are $3^2$ choices. 
%If $\lambda = 1$ then there are $2^2$ choices. 
%Finally if $\lambda = 2$ then there are $1^2$ choices. 
%So there are $\sum_{i=1}^5 i^2$ different possibilities for defining peak and valley preserving automata when $\alpha =2$. 
%\item Suppose  $\alpha = 1$ Similarly to the first case, 
%there are $\sum_{i=1}^4 i^2$ choices. 
%\item If  $\alpha = 0$: 
% there are $\sum_{i=1}^3 i^2$ choices. 
%\item If $\alpha = -1$: there are $\sum_{i=1}^2 i^2$ choices. 
%\item If  $\alpha = -2$: there are $\sum_{i=1}^1 i^2$ choices. 
%\end{enumerate}

\section{Equicontinuity and points of equicontinuity}\label{equicontinuity}
In this section we investigate the equicontinuity of radius one sand automata.
 In
\cite{cfm} and \cite{dgm}, the authors classify one dimensional sand automata as: either sensitive, or  nonsensitive without an
equicontinuity point, or non-equicontinuous with an equicontinuity
point, or finally equicontinuous. 

\subsection{Vertical inducing points and equicontinuity}\label{vertical_inducing}

    Given that
a sand automaton is topologically conjugate to a 2-dimensional
cellular automaton, we use the
following result:

\begin{theorem}[Proposition 3.14, \cite{dgm}]
$\Phi $ is equicontinuous if and only if  $\Phi$ is ultimately periodic, if and only if 
$ \forall {\bf x} \in X,$ $ \Phi({\bf x})$ is eventually periodic.
\label{equicontinuous1}
\end{theorem}
In order to classify $\Gamma$ we introduce the following definitions.
Let $n \in \mathbb{N}$.  A configuration ${\bf x}$ is a \emph{vertical
  inducing point} of order n for a sand automaton $\Phi $ if
$\Phi({\bf x}) = \rho^n({\bf x})$.
 For example, a fixed
point for $\Phi$ is also a vertical inducing point of order 0.
\begin{lemma}
If ${\bf x}$ is a vertical inducing point of order n, then $\Phi^m({\bf x}) = \rho^{mn}({\bf x}),$ for each  $m \in \mathbb{N}$. Also
 $e(x)\in  S_{{{1}\choose{0}}}$ satisfies  $\Phi_{{{1}\choose{0}}}^m({\bf x}) = \sigma_V^{mn}({\bf x}), \,$ for each $ m \in \mathbb{N}$.

\label{VIP1}
\end{lemma}

\begin{proof}
If $m = 1$, then $ \Phi^1({\bf x}) = \rho^n({\bf x})$ by definition.
Suppose that for $m=k$, $\Phi^k({\bf x}) = \rho^{kn}({\bf x})$.  
Then \[\Phi^{k+1}({\bf x}) = \Phi(\Phi^k({\bf x})) = \Phi(\rho^{kn}({\bf
  x})) = \rho^{kn}(\Phi({\bf x})) = \rho^{kn}(\rho^n({\bf x})) =
\rho^{(k+1)n}({\bf x}).\]
Also, for each $m$, $\Phi_{{{1}\choose{0}}}^{m}(e({\bf x})) = e(\Phi^{m}({\bf x}))= e(\rho^{nm}({\bf x})) = \sigma_V^{mn}({\bf x})\,$. 
\end{proof}

\begin{corollary}
\label{vertical_implies_not_equicontinuous}
If a sand automaton $\Phi$ admits a  vertical inducing point of nonzero order, then $\Phi_{{{1}\choose{0}}}$ is not equicontinuous. 

\end{corollary} 

\begin{proof}
Let $n \in \mathbb{N}$.  By Lemma \ref{VIP1} if \textbf{x} is a vertical inducing point of
order $n\neq 0$ then $\Phi_{{{1}\choose{0}}}^m(e(\textbf{x})) =
\sigma_V^{mn}(\textbf{x}), \, \forall m \in \mathbb{N}$.  The only way that 
$e(\textbf{x}) $ is eventually periodic is if ${\bf x}$ consists entirely of sinks  or sources; however in this case  ${\bf x}$ would be vertical inducing of order $0$. 
Theorem \ref{VIP1} now implies the result; the fact that equicontinuity is a topological property implies the second statement.
\end{proof}

In the next theorem we identify a class of
sand automata that have vertical inducing points.

\begin{theorem}
\label{nonzerovalue}
Let $\Phi$ be a radius one sand automaton. If there exists a nonzero $m$ such that $m$ appears in each row, or each column, of $\Phi$'s local rule table, then $\Phi$ admits a vertical inducing point of nonzero order.
\end{theorem} 

\begin{proof}
Suppose that the nonzero value m occurs in every row of the local rule
table for the sand automaton $\Phi$.  Note that if $f\circ
\Pi(x) = ({{L_i}\choose{R_i}})_{i\in \mathbb Z}$, then $
L_{i+1}=-R_{i} $ for each $i$. Conversely, if
$({{L_i}\choose{R_i}})_{i\in \mathbb Z}\,\in f(\Pi(X))$, is such that
each $ L_{i+1}=-R_{i} $, then there is some ${\bf x} \in X$ such that
$f(\Pi(x)) = ({{L_i}\choose{R_i}})_{i\in \mathbb Z}\,\in f(\Pi(X))$.
To find a vertical inducing point then, it is sufficient to find a
cycle ${{L_1}\choose{R_1}},\ldots , {{L_k}\choose{R_k}}$ where
$L_{i+1} = -R_i, \,$ for $1\leq i <k$, $L_1 = -R_k$ and such that
$\phi{{L_i}\choose{R_i}} = m, \,$ for each $1\leq i \leq k$.  First
select any ${{L_1}\choose{R_1}}$ such that $\phi(L_1,R_1) = m$.  If
$R_1 = -L_1$ then we are done.  If not, select ${{L_2}\choose{R_2}}$
such that $L_2 = -R_1$ and $\phi(L_2,R_2) = m $ - this this can be
done because the local rule table has the value m in every row. If
$R_2 = -L_1 $ or $ R_2 = -L_2$ then we are done; if not continue this
process until we find ${{L_1}\choose{R_1}}\ldots {{L_k}\choose{R_k}}$
such that for some $1\leq j\leq k$, $L_{j}=-R_{k}$ - this has to occur since the gradient pair entries come from a finite alphabet. The desired
cycle is ${{L_j}\choose{R_j}}\ldots {{L_k}\choose{R_k}}$.  The proof
when there is a nonzero m in every column is similar except that newly
selected gradient pairs will come before the current gradient pairs.
\end{proof}

Recall that 
an $n \times n$ {\em latin square} is a table filled using an alphabet $\mathcal A$ of size $n$, where each row and each column has exactly one occurrence of each member of $\mathcal A$.   Theorem 
\ref{nonzerovalue} and Corollary \ref{vertical_implies_not_equicontinuous} then imply the following:

\begin{corollary}
If a radius one sand automaton $\Phi$ has a Latin square local rule table then $\Phi $ is not equicontinuous. 
\label{LatinSquare}
\end{corollary}

 Let $\Phi $ be a radius 1 linear sand automaton.
Then the local rule for $\Phi$, denoted by $\phi$, is a homomorphism on
a cyclic group with group operation addition $\oplus$.  In particular
this implies that $\phi(L,R) = \alpha L
\oplus \beta R $, where $\alpha $ and $\beta \in \mathbb{Z}_5 $. This
implies  that $\Phi$'s local rule table is either a latin square, or has a nonzero element on the anti-diagonal.
\begin{proposition}
Let $\Phi $ be a radius 1 linear sand automaton.  
Then $\Phi$ admits a vertical inducing point, and so is not equicontinuous.
\label{generalEqiu}
\end{proposition}

\begin{proof}
It is sufficient to show that there is a word $ {{L_0}\choose{R_0}}
\ldots {{L_n}\choose{R_n}}$ in $(\mathbb Z_{5}^{2})^{+} $ such that
$L_{i+1}=-R_{i}$ for $i=1,\ldots k-1$, $R_n = -L_0$ and there exists a
nonzero $m$ such that for each $i$, $\gamma(L_i, R_i)= m$.  Suppose
first that there is a nonzero value on the anti-diagonal of the local
rule table for $\Phi$.  Then there is a gradient pair
${{a}\choose{-a}}$ that returns a nonzero value m under $\phi$. In this case, the
word ${{a}\choose{-a}}$ is the desired one. If all anti-diagonal rule entries are 0, then for each $L\in \mathbb Z_{5}$, $ \alpha L \oplus \beta (-L) =0$.  This implies that $\alpha L = \beta L$, so that $ \alpha
= \beta$.  Thus $\phi =\alpha L \oplus \alpha
R = \alpha(L \oplus R) $.  But the local rule $L \oplus R$ corresponds
to the linear sand automaton $\Gamma$ which has a Latin square as a
local rule table.  Therefore $ \alpha(L \oplus R) $'s rule table  is a permutation
of this local rule table and so also a
Latin square. Corollary \ref{LatinSquare} now yields the result.
\end{proof}

Note that linearity is not used in the case where there is a nonzero value on the anti-diagonal of the local rule table. Thus we can make  the following statement for general radius one sand automata:

\begin{proposition}
If a radius one sand automaton $\Phi$ has a nonzero value on the anti-diagonal 
of its local rule table then $\Phi$ has a vertical inducing point.  
\label{generalVIP1}
\end{proposition}

 There are  $5^{25}$  radius one sand automata, and since there are $5^{20}$
local rule tables with only zero entries on the anti-diagonal, at least $99.968\%$ of the total number of radius 1 sand
automata have a vertical inducing point,  and so are not equicontinuous.
%In a method similar to that of Section \ref{vertical_inducing}, we can
%look for less simple vertical inducing points using a graph like
%$\mathcal H$. A counting argument then shows that at least $99.968 +
%0.00144553728 \doteq 99.969 \%$ of radius 1 sand automata are not
%equicontinuous.

 An example of a
vertical inducing point for $\Gamma$ is $ \textbf{x} = \overline{0,\,1,\,3,\,1,\,0} \cdot
\overline{0,\,1,\,3,\,1,\,0} $. Then $f(\Pi({\bf x})) =    \overline{{{0}\choose{1}},
         {{-1}\choose{2}},\, {{-2}\choose{-2}},\, {{2}\choose{-1}},\,
         {{1}\choose{0}}}\cdot \overline{{{0}\choose{1}},\,
         {{-1}\choose{2}},\, {{-2}\choose{-2}},\, {{2}\choose{-1}},\,
         {{1}\choose{0}}}$,  so that  $\Gamma(\textbf{x}) = \textbf{x} +
         \overline{1}\cdot \overline{1} \neq \textbf{x}$.  In fact
         the sand automaton $\Gamma$ has an infinite number of
         vertical inducing points for each $-2 \leq n \leq 2$. We discuss this in the following section.
First though we show that while $\Gamma$ is not equicontinuous, it does  have equicontinuity points, putting it in Category (3) of the classification in \cite{dgm}. 
We say that a word ${\bf w}$ is \emph{blocking} for a sand automaton
$\Phi$ if there $\exists \, k,s \in \mathbb{N}$ such that $ 0 \leq k+s
\leq |{\bf w}|$ and such that whenever $  {\bf x}$ and $ {\bf y} \in
     [{\bf w}]_i $, then  $(\Phi^n({\bf x}))_{[i+k, i+|{\bf w}|-s]} =
     (\Phi^n({\bf y}))_{[i+k, i+|{\bf w}|-s]} $ for all natural $n$.

\begin{proposition}
Let ${\bf w} = 0 \, 3 \, 2 \, 3\, 0$.  Then ${\bf w}$ is blocking for the sand automaton $\Gamma$. Thus $\Gamma$ has equicontinuity points. 
\label{blockingw1}
\end{proposition}
\begin{proof}
First note that the first statement will imply the second.
Note also  that $f\circ \Pi({\bf w}) = {{L_0}\choose{2}},\,{{-2}\choose{-1}},\,{{1}\choose{1}},\,{{-1}\choose{-2}},\,{{2}\choose{R_4}}$.  This block satisfies
 $\gamma ({{-2}\choose{-1}},\,{{1}\choose{1}},\,{{-1}\choose{-2}})= 
(2,\,2,\,2) $.  
Suppose that ${\bf x}, {\bf y} \in [{\bf w}]_{-2}$.  
Then
$(\Gamma({\bf x}))_{[-2,2]} = (x_{-2}^{\prime}, \, 5, \, 4, \, 5, \,
x_{2}^{\prime})$ and  $ (\Gamma({\bf y}))_{[-2,2]} = (y_{-2}^{\prime}, \, 5, \,
4, \, 5, \, y_{2}^{\prime})$, where $x_{i}^{\prime},y_{i}^{\prime} \leq
2, \,$ for  $i\, \in \{-2,2\}$.  Under the action of $\Gamma$ the central three cells
increase by 2 and the right most and left most cells can increase by
at most 2.  This implies that $ (f\circ \Pi(\Gamma({\bf x})))_{[-1,1]} $ and
$(f\circ \Pi(\Gamma({\bf y})))_{[-1,1]}$ equal  $(f \circ
\Pi ({\bf x}_{[-1,1]}))$, and $\Gamma$ thus adds 2 to these three positions. An inductive argument completes the proof. 
\end{proof}

There are 24 nontrivial homomorphisms $f: {\mathbb
  Z}_{5}^{2}\rightarrow {\mathbb Z}_{5}$.  If $f(L,R)= \alpha L \oplus
\beta R$ we use the notation $(\alpha,\beta)$ to represent $f$.  Eight
of these maps have blocking words of the type described in Proposition \ref{blockingw1}. These are $(a,a)$, where $ a \neq 0$, and
$(-1,2)$, $(1,-2)$, $(2,-1)$, $(-2,1)$.

\subsection{Local-rule-constant configurations}\label{local_rule_constant}

The concept of a vertical inducing point is a special case of a more
general type of configuration ${\bf x}$ where $(\Phi^n({\bf x}))_{n}$
is easily computable. Let us say that sand automaton $\Phi$ is {\em
  local-rule-constant} at ${\bf x}$ (or ${\bf x}$ is a $\Phi$
local-rule-constant point) if for some ${\bf y}\in X$, $\Phi^n({\bf
  x}) = {\bf x} + n\textbf{y}$ for each $ n \in \mathbb{N}$.  For example, if ${\bf
  x}$ is a vertical inducing point, then it is a $\Phi$
local-rule-constant configuration, with ${\bf y}$ constant.  However
the set of local-rule-constant configurations is much larger; we now describe a family of these points for $\Gamma$. 

Let $w=w_{1}\ldots w_{k} = {{L_1}\choose{R_1}} \ldots
{{L_k}\choose{R_k}}\, \in ({\mathbb Z}_{5}^{2})^{+}$ satisfy the
properties that $R_{i}=-L_{i+1}$ for $i=1, \ldots k-1$, and $\gamma_{*}
(w_{i})$ is constant $i=1, \ldots k$; then we say that $w$ is a {\em
  cycle segment}. In Table \ref{xplusny2}, we list some  cycle
segments for $\Gamma$. Consider the directed graph $\mathcal H$ whose
vertices are the cycle segments for $\Gamma$ listed in Table \ref{xplusny2}, and such that there is
an edge from vertex $V$ to vertex $V'$ if and only if the following are satisfied:
\begin{enumerate}
\item If the cycle segment corresponding to $V$ ends with a gradient pair
  ${{*}\choose{a}}$ then the  cycle segment corresponding to $V'$ starts
  with a gradient pair ${{-a}\choose{\star}} $.
\item If the cycle segment corresponding to $V$ ends with a
  gradient pair $ {{*}\choose{2}}$ and has order j, then the order of
  the cycle segment corresponding $V'$  must have order at least 
  $j$.
\item If the cycle segment corresponding to $V$ ends with a gradient pair $
  {{*}\choose{-2}}$ and has order j, then the cycle segment corresponding to 
$V'$  must have order at most $j$.
\end{enumerate}

Note that vertices $K$, $L$ and $M$ in $\mathcal H$ are isolated. Define ${\mathcal
  G}^{*}$ to be the set of all configurations ${\bf x}$ in $X$ such
that $f(\Pi({\bf x}))$ corresponds to an infinite path in $\mathcal
H$. In this context the infinite loops at $K$, $L$ and $M$ correspond
to $\Gamma$-fixed points in $X$.

\begin{theorem}
If ${\bf x} \in \mathcal G^{*}$,  then ${\bf x} $ is
$\Gamma$ local-rule-constant.
\end{theorem}
\begin{proof}

 Choose an ${\bf x} \in X$ such
that $f \circ \Pi ({\bf x}) $ is represented by an infinite
path $ {\bf V} =\ldots V_{-2}\,V_{-1}\, \cdot
V_0\,V_1\,V_2\,\ldots \in \mathcal H$. Let ${\bf y}$ be the point in $X$ obtained by applying $\gamma_{*}$ to the representative in $f(\Pi(X))$ of $
{\bf V} $.  We claim that
$\Gamma^n({\bf x}) = {\bf x} +n{\bf y}$.  Suppose that  $V_{j}$ corresponds to the cycle segment  $(f(\Pi(x)))_{i_{j-1}+1}, \ldots (f(\Pi(x)))_{i_{j}}$, and it ends 
 with a gradient pair of the form ${{**}\choose{
    2}}$. Then in ${\bf x}$, $x_{i_{j}+1} -x_{i_{j}} \geq 2$. 
Geographically
speaking there is a ``steep hill" to the right of $x_{i_j}$. By condition
(2), $V_i$ can only be followed by an $V_{i+1}$ whose corresponding cycle segment  has order
at least that of $V_i$'s.  Therefore in $\Gamma({\bf x})$
the ``steep hill", if it changes, can only get steeper.
 Thus
$f(\Pi(\Gamma({\bf x})))_{i_j} = f(\Pi({\bf x}))_{i_j}$.
 Similarly if
$V_i$ ends with  a gradient pair of the form ${{**}\choose{ -2}}$, Condition
(iii) guarantees that if $f(\Pi({\bf x}))_{i_j} = -2 $ then
$f(\Pi(\Gamma({\bf x})))_{i_j} = -2$. This fact is true for all $j$.
    Finally if ${\bf V}$ corresponds to the infinite loop at  $K$, $L$ or $M$,
 then
$f(\Pi({\bf x}))$ is constant and ${\bf y} = \overline{0}\cdot
\overline{0}$, so that $\Gamma({\bf x}) = {\bf x}$ in which case
$\Gamma$ is (trivially) local-rule-constant.  Thus $f \circ \Pi (\Gamma{\bf x}) $ is also represented by ${\bf V}$.
By induction  it follows that ${\bf x}$ is
$\Gamma$-local-rule-constant.
\end{proof}

  For example, suppose that we want a configuration ${\bf x}$ that
  under the action of $\Gamma$ we have $\Gamma^n({\bf x}) = {\bf x} +
  n{\bf y}$ where $ {\bf y} = \overline{2},\cdot 1,1,1,1,\overline{2}
  $.  Then to build such a configuration using the cycles from the
  vertical inducing points we can let $f\circ \Pi({\bf x}) =
  \overline{{{2}\choose{0}},{{0}\choose{2}},{{-2}\choose{-1}},{{1}\choose{1}},{{-1}\choose{-2}}
  }\cdot
           {{2}\choose{-1}},{{1}\choose{0}},{{0}\choose{1}},{{-1}\choose{2}},
           \overline{{{-2}\choose{-1}},{{1}\choose{1}},{{-1}\choose{-2}},{{2}\choose{0}},{{0}\choose{2}}
           }$.  This point leads to an infinite number of
           configurations in the pre-image set $(f \circ
           \Pi({\bf x}))^{-1}$.  One such point is ${\bf x} =
           \overline{2,2,4,3,4},\cdot 2,1,1,2,\overline{5,4,5,3,3}$.

\begin{table}[h]
\begin{center}
\begin{tabular}{|c|c|c|}
\hline
Label & Cycle segment & Order \\ \hline
A  & ${{2}\choose{-1}}{{1}\choose{0}}{{0}\choose{1}}{{-1}\choose{2}}$ & 1 \\ \hline
B  & ${{-2}\choose{-2}} $ & 1 \\ \hline
C  & ${{2}\choose{0}}{{0}\choose{2}}$ & 2 \\ \hline
D & ${{-2}\choose{-1}}{{1}\choose{1}}{{-1}\choose{-2}}$& 2 \\ \hline
E  & ${{2}\choose{-2}} $ &0 \\ \hline
F  & ${{-2}\choose{2}}$ & 0\\ \hline
G  & $ {{-2}\choose{0}}{{0}\choose{-2}} $& -2 \\ \hline
H  & $ {{2}\choose{1}}{{-1}\choose{-1}}{{1}\choose{2}}$& -2 \\ \hline
I  & $ {{-2}\choose{1}}{{-1}\choose{0}}{{0}\choose{-1}}{{1}\choose{-2}}$ &-1  \\ \hline
J  & $ {{2}\choose{2}} $ & -1 \\ \hline
K  & ${{-1}\choose{1}} $ & 0\\ \hline
L  & ${{1}\choose{-1}} $ & 0\\ \hline
M  & ${{0}\choose{0}} $ & 0\\ \hline
\end{tabular}
\end{center} 
\caption{All of the cycle segments that can be used to create points $\Gamma^n({\bf x})= {\bf x} +n{\bf y}$.}
\label{xplusny2}
\end{table}
\begin{figure}[h]
\centerline{\includegraphics[scale =0.2]{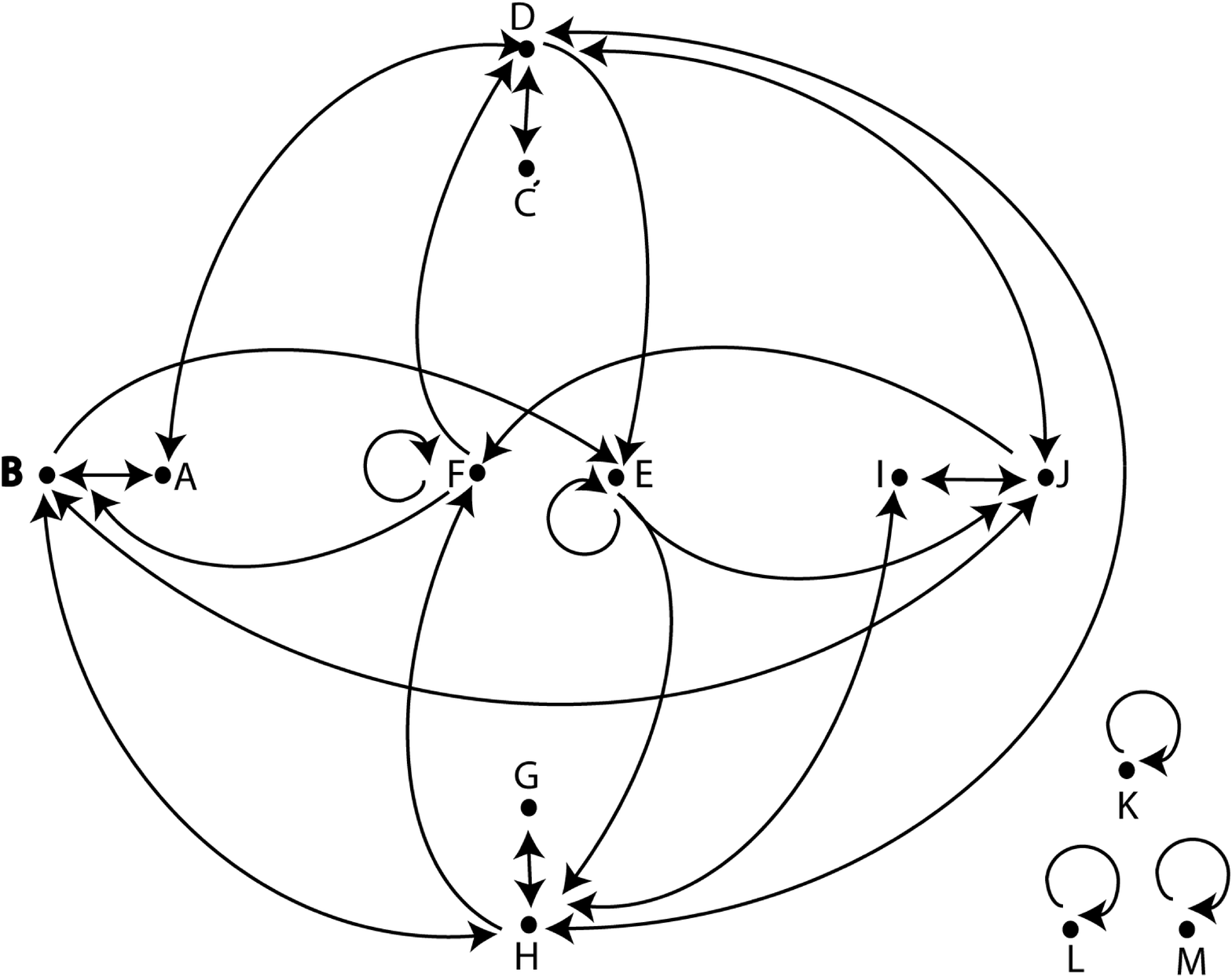}}
\caption{The graph for the set of all points made up of constant
  blocks which will add a consistent ${\bf y}$ at each time step. See
  Table \ref{xplusny2} for definitions of A,B,C,D,E,F,G,H,I,J,K,L,M.}
\label{graph2}
\end{figure}
\vspace{6pt}

 Note that the set
$\mathcal G$ defined in Section \ref{surjective} is contained in $\mathcal G^{*}$.
Note also  that $\mathcal G^{*}$ is
closed and 
 $\Gamma$-invariant.  
An
interesting question is whether $\mathcal G^{*}$ is an
attractor set for the sand automaton $\Gamma$.  We have conducted
simulations where the space time diagrams of several initial
configurations are generated.  Empirically what seems to be happening
is that the iterates of the initial configuration converge ``almost
everywhere" to a configuration in $\mathcal G^{*}$.  We describe what we mean by this:
define the words  $O=  {{-2}\choose{-2}}{{2}\choose{0}}{{0}\choose{2}}{{-2}\choose{-2}}, 
P=  {{-2}\choose{-1}}{{1}\choose{0}}{{0}\choose{1}}{{-1}\choose{-2}}$ and
$Q= {{2}\choose{2}}{{-2}\choose{0}}{{0}\choose{-2}}{{2}\choose{2}},
R= {{2}\choose{1}}{{-1}\choose{0}}{{0}\choose{-1}}{{1}\choose{2}}$.  
These sets of words are {\em 2-periodic} in the sense that if 
$f \circ \Pi (\Phi^n(\textbf{x})_{[i,i+3]}) = O$ then $f \circ \Pi (\Phi^{n+1}(\textbf{x})_{[i,i+3]}) = P$ and if $f \circ \Pi (\Phi^n(\textbf{x})_{[i,i+3]}) = Q$ then $f \circ \Pi (\Phi^{n+1}(\textbf{x})_{[i,i+3]}) = R$.   If we include these in a new graph ${\mathcal H}'$, then this seems to describes the asymptotic behaviour of $\Phi$ more accurately.  The graph ${\mathcal H}^{\prime}$ is presented in Figure \ref{graph3}.
This leads to the following conjecture. Similar to our definition of ${\mathcal G}^{*}$, let
\[\mathcal G': = \{ 
{\bf x}: f(\Pi({\bf x})) 
\mbox{ corresponds to an infinite path in }{\mathcal H }'
\}.\]

\begin{conjecture}
The set $ \mathcal G'$ is an attractor for $\Gamma$, in that if ${\bf x}\in X$, then $\lim_{n\rightarrow \infty}d(\Gamma^{n}(x), \mathcal G')= 0$. 
\end{conjecture}

\begin{figure}[h]
\centerline{\includegraphics[scale=0.3]{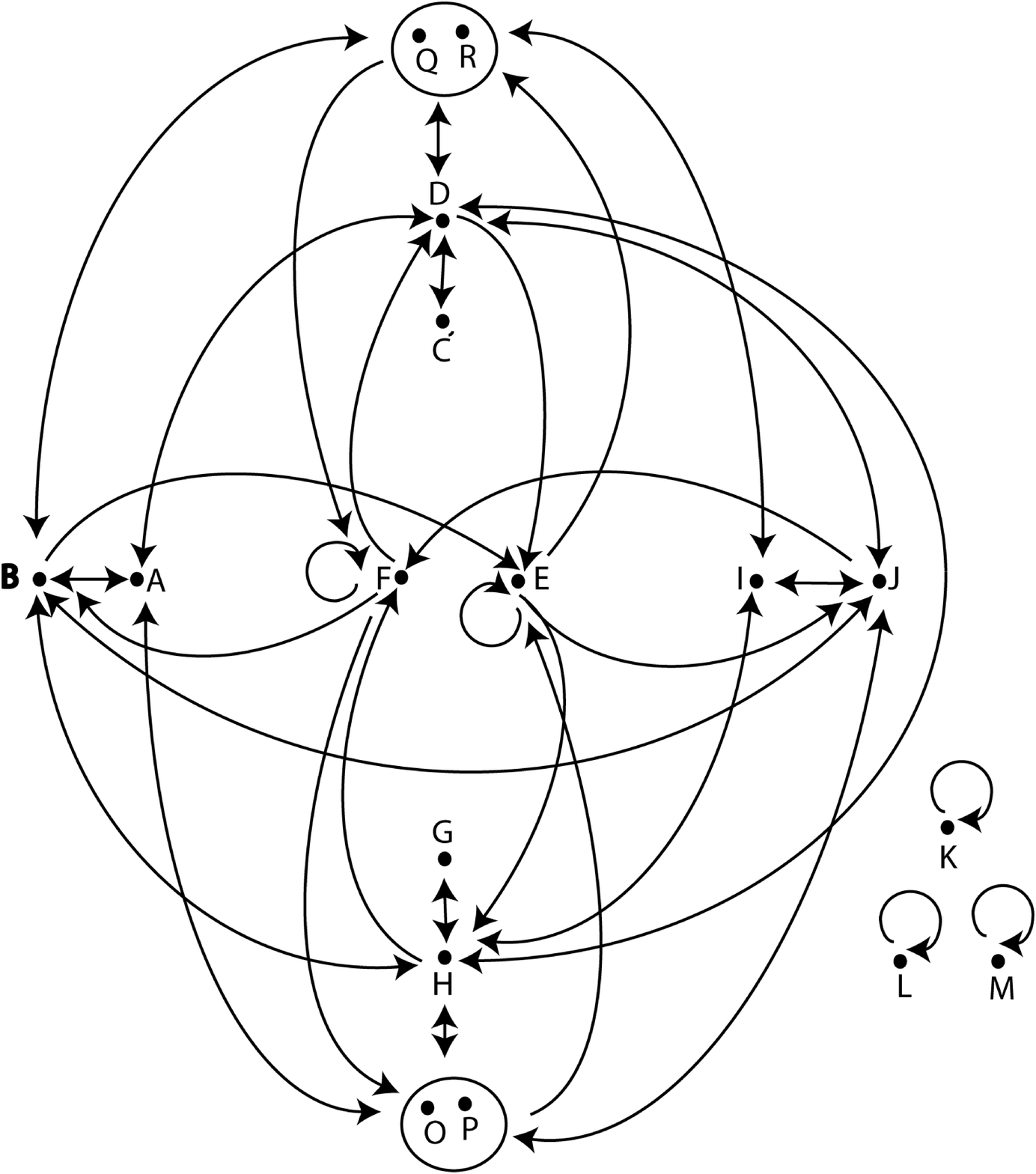}}
\caption{A potential attractor set for $\Gamma$.}
\label{graph3}
\end{figure}

\begin{figure}[htb]
\includegraphics[scale=0.15]{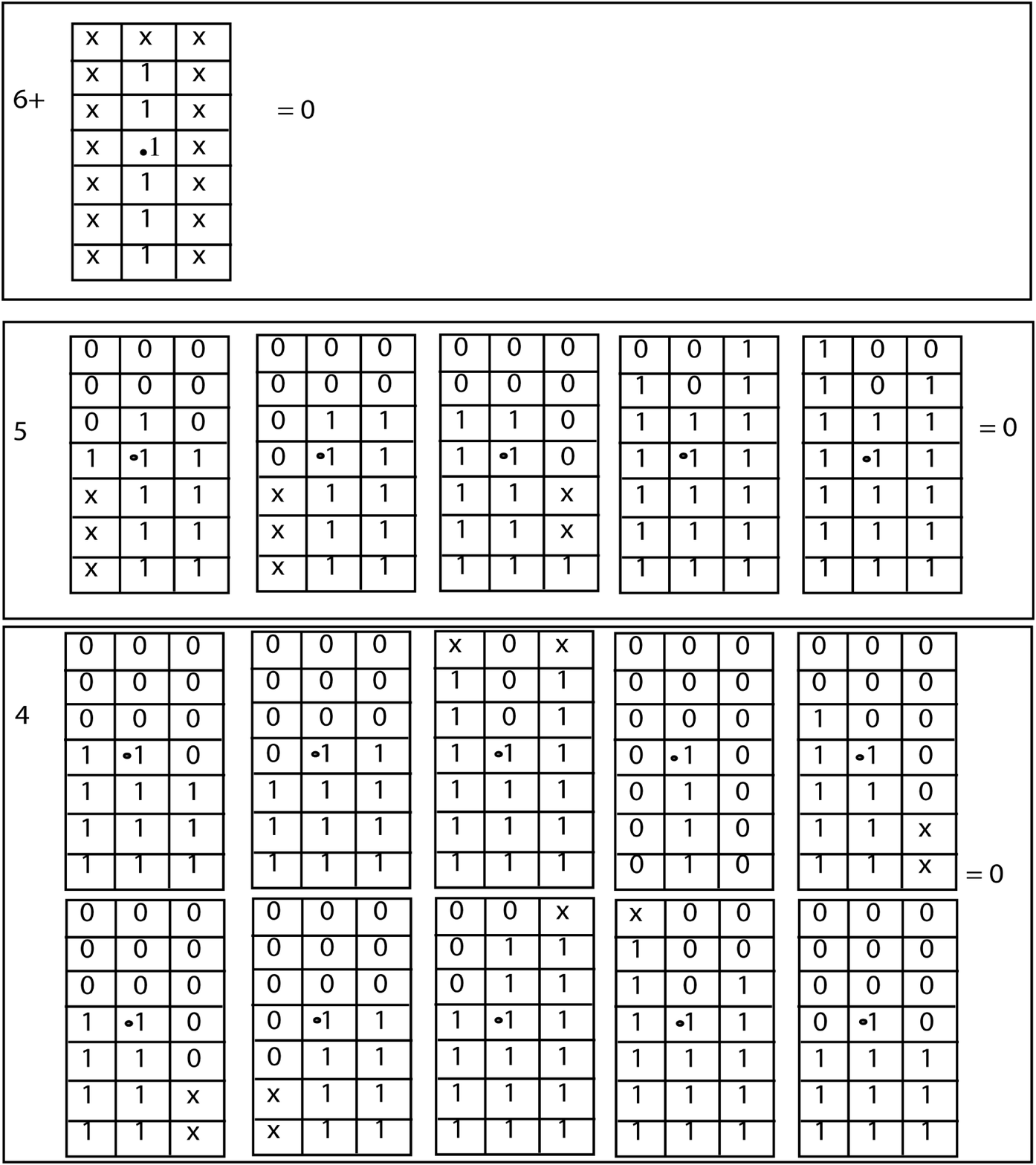} \\
\includegraphics[scale= 0.15]{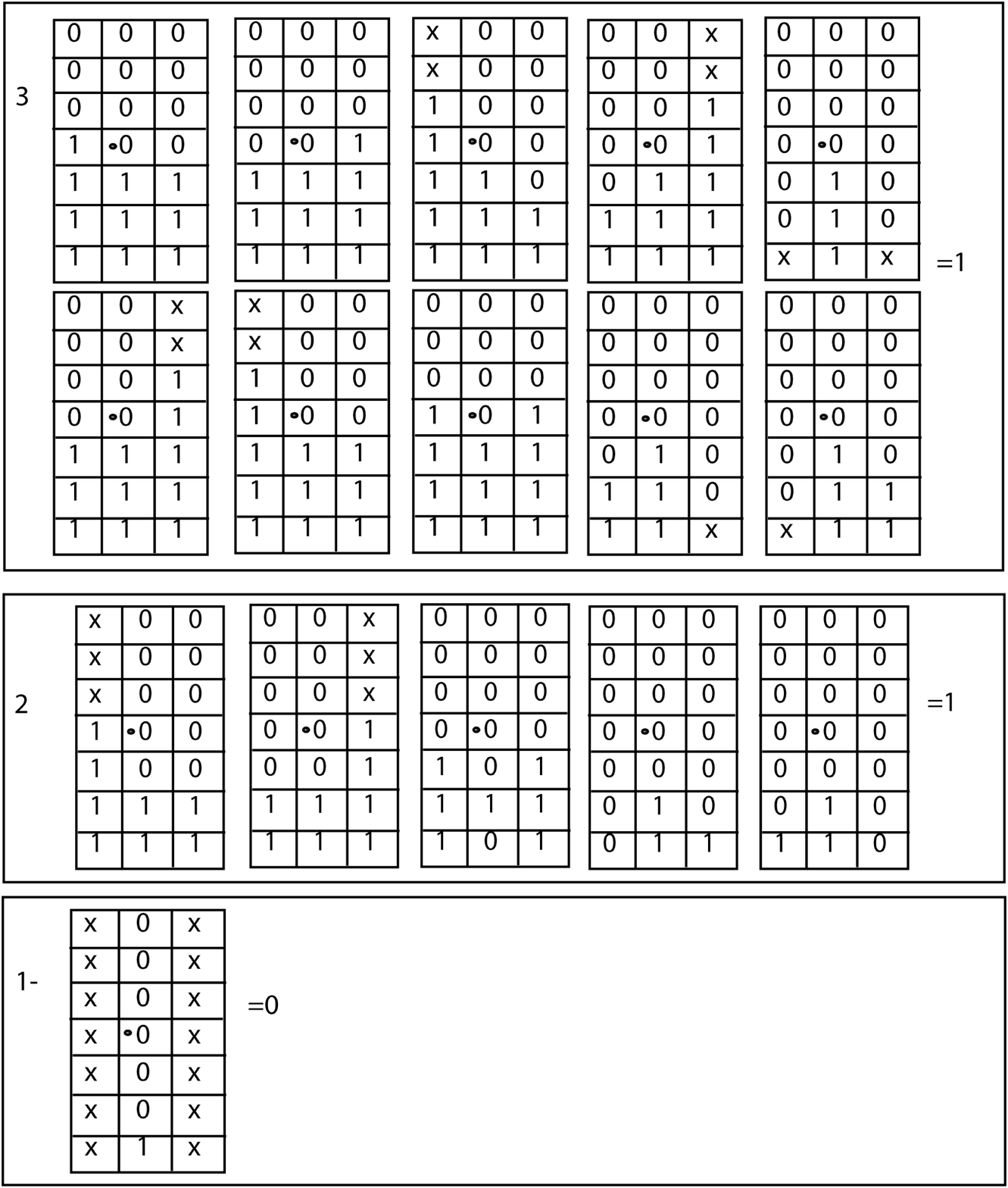}
\caption{The local rule table for $\Gamma_{{{1}\choose{0}}}$.  The number on the left of
  each rectangle represents the number of ones in the reference
  column.  An `x'
  represents a cell that can have any value.  The $\cdot$
  represents the central cell.  Only configurations where the central
  cell changes are listed.}
\label{SAtoCA2}
\end{figure}

{\footnotesize
\bibliographystyle{alpha}
\bibliography{bibliography}

\newcommand{\etalchar}[1]{$^{#1}$}
\def\ocirc#1{\ifmmode\setbox0=\hbox{$#1$}\dimen0=\ht0 \advance\dimen0
  by1pt\rlap{\hbox to\wd0{\hss\raise\dimen0
  \hbox{\hskip.2em$\scriptscriptstyle\circ$}\hss}}#1\else {\accent"17 #1}\fi}
\begin{thebibliography}{CFMM97}

\bibitem[BBC{\etalchar{+}}92]{bbcds}
J.~Banks, J.~Brooks, G.~Cairns, G.~Davis, and P.~Stacey.
\newblock On {D}evaney's definition of chaos.
\newblock {\em Amer. Math. Monthly}, 99(4):332--334, 1992.

\bibitem[BTW88]{btw}
Per Bak, Chao Tang, and Kurt Wiesenfeld.
\newblock Self-organized criticality.
\newblock {\em Phys. Rev. A (3)}, 38(1):364--374, 1988.

\bibitem[CFM07]{cfm}
Julien Cervelle, Enrico Formenti, and Beno{\^{\i}}t Masson.
\newblock From sandpiles to sand automata.
\newblock {\em Theoret. Comput. Sci.}, 381(1-3):1--28, 2007.

\bibitem[CFMM97]{cfmm}
Gianpiero Cattaneo, Enrico Formenti, Giovanni Manzini, and Luciano Margara.
\newblock On ergodic linear cellular automata over {$Z_m$}.
\newblock In {\em S{TACS} 97 ({L}\"ubeck)}, volume 1200 of {\em Lecture Notes
  in Comput. Sci.}, pages 427--438. Springer, Berlin, 1997.

\bibitem[CM96]{cm}
Bruno Codenotti and Luciano Margara.
\newblock Transitive cellular automata are sensitive.
\newblock {\em Amer. Math. Monthly}, 103(1):58--62, 1996.

\bibitem[DGM09]{dgm}
Alberto Dennunzio, Pierre Guillon, and Beno{\^{\i}}t Masson.
\newblock Sand automata as cellular automata.
\newblock {\em Theoret. Comput. Sci.}, 410(38-40):3962--3974, 2009.

\bibitem[Hed69]{h}
G.~A. Hedlund.
\newblock Endormorphisms and automorphisms of the shift dynamical system.
\newblock {\em Math. Systems Theory}, 3:320--375, 1969.

\bibitem[I{\^O}N83]{ino}
Masanobu It{\^o}, Nobuyasu {\^O}sato, and Masakazu Nasu.
\newblock Linear cellular automata over {$Z_{m}$}.
\newblock {\em J. Comput. System Sci.}, 27(1):125--140, 1983.

\bibitem[MM98]{mm}
Giovanni Manzini and Luciano Margara.
\newblock Invertible linear cellular automata over {${\bf Z}_m$}: algorithmic
  and dynamical aspects.
\newblock {\em J. Comput. System Sci.}, 56(1):60--67, 1998.

\end{thebibliography}
}

\end{document}